\documentclass{amsart}
\usepackage{articletemplate}

\numberwithin{equation}{section}

\usepackage{lineno}

\renewcommand\LineNumber{\the\inputlineno}

\begin{document}
\title[On the distribution of additive twists of the divisor function]{On the distribution of additive twists of the divisor function and Hecke eigenvalues}
\author{mayank pandey}
\address{Department of Mathematics, Princeton University, Princeton, NJ 08540, USA}
\email{mayankpandey9973@gmail.com}
\maketitle

%
%
%
\section{Introduction}
Let $f$ be an $\SL_2(\ZZ)$ cusp form of weight $k$, and suppose it has Fourier expansion
\[f(z) = \sum_{n\ge 1} \lambda_f(n)n^{\frac{k - 1}{2}}e(nz)\]
for $z$ in the upper half plane. In this paper, one of our main objects of interest will be the exponential sum 
\[S_f(\alpha; X) = \sum_{n\le X}\lambda_f(n)e(n\alpha).\]
Jutila \cite{J} showed that this sum is $O(\sqrt{X})$ uniformly in $\alpha$ and therefore exhibits considerable oscillation. 
By Plancherel and (14.56) in \cite{IK} 
\begin{equation}
    \int_0^1 |S_f(\alpha)|^2d\alpha = \sum_{n\le X} |\lambda_f(n)|^2 = c_fX + O(X^{3/5})
    \label{eq:rankin_selberg}
\end{equation}
for some $c_f > 0$, so it is clear that Jutila's bound is sharp. 
By H\"older's inequality, it follows from these estimates that for all $s > 0$
\[\int_0^1 \bigg|\sum_{n\le X}\lambda_f(n)e(n\alpha)\bigg|^{s} d\alpha\asymp X^{\frac s2}.\]
It is desirable to know whether one can determine more information about the distribution of this exponential sum. 

Another related exponential sum is 
\[
    S_d(\alpha; X) = \sum_{n\le X} d(n)e(n\alpha)
\]
where
\[
    d(n) = \sum_{d | n} 1
\]
is the divisor function.

The properties of this exponential sum are of interest in applications of the circle method. It behaves differently from the 
exponential sum $S_f(\alpha; X)$ due to the positivity of its coefficients.
For $s > 2$, due to this positivity, the contribution of $\alpha$ near $0$ 
(and more generally near rationals with small denominator) determine the size of the $L^s$ norm. 
Using the circle method, finding asymptotics of the form 
\[
    \int_0^1 |S_d(\alpha; X)|^s d\alpha\sim C_sX^{s - 1}(\log X)^s
\]
is then quite straightforward. See \cite{P} for a proof of this for higher divisor functions, where the situation is similar.
Asymptotics in the case $s = 2$ quickly follow from Plancherel, as we have 
\[
    \int_0^1 |S_d(\alpha; X)|^2 d\alpha = \sum_{n\le X}d(n)^2\sim \frac{1}{\pi^2}X(\log X)^3.
\]
Lower moments are significantly more difficult, as for $s < 2$ one expects a nontrivial contribution from the minor arcs as well. 
Until now, the only result for moments in this range was for the $L^1$-norm and due to Goldston and the author in \cite{GP}, where it is shown that
\[
    \sqrt{X}\ll\int_0^1 |S_d(\alpha)|d\alpha\ll \sqrt{X}\log X.
\]
In this paper, we are able to find asymptotics for the $L^s$-norm of $S_d(\alpha; X)$ for all $0 < s < 2$. 
This, combined with the aforementioned results for higher moments resolves the problem of finding asymptotics for all moments of $S_d(\alpha; X)$. 
Using the same method, we are able to show the following similar result for all moments of $S_f(\alpha; X)$ as well in Theorem \ref{thm:main_thm}.
\begin{theorem}
    For $0 < s < 2$, we have that for $\star\in\set{d, f}$, with $f$ a holomorphic cusp form 
    for $\SL_2(\ZZ)$
    \[
        \int_0^1 |S_\star(\alpha; X)|^sd\alpha = C_s^\star X^{\frac s2} + O(X^{\frac{s}{2} - \eta_1s(2 - s)})
    \]
    for some $C_s^\star, \eta_1 > 0$. Furthermore, for $s\ge 2$, we also have that
    \[
        \int_0^1 |S_f(\alpha; X)|^sd\alpha = C_s^fX^\frac{s}{2} + O_{s, f}(X^{\frac s2 - \eta_2})
    \]
    for some $\eta_2 > 0$.
    \label{thm:main_thm}
\end{theorem}
This result has several interesting corollaries in the case $\star = f$. Note that from the bound
$S_f(\alpha; X)\ll\sqrt{X}$, (\ref{eq:rankin_selberg}), and H\"older, we have that $\exp(-C_1s)\le C_s^f\le\exp(C_2s)$ for some $C_1, C_2 > 0$.
Then, by the method of moments (one may apply Theorem 9.2 in \cite{Gu}, for example), we obtain a limiting distribution for the magnitude of $\frac{1}{\sqrt{X}}S_f(\alpha; X)$ sum as $X\to\infty$.
\begin{corollary}
    Suppose $f$ is a holomorphic cusp form for $\SL_2(\ZZ)$. 
    Let $A_{X}$ be the random variable given by 
    \[
        \frac{1}{\sqrt{X}}\bigg|\sum_{n\le X} \lambda_f(n)e(n\alpha)\bigg|,
    \] 
    where $\alpha$ is chosen uniformly at random from $[0, 1]$. Then, there is a random variable 
    $A$ so that $A_X$ converges to $A$ in distribution as $X\to\infty$. 

    In particular, it follows that there exists a compactly supported measure $\mu$ on $[0, \infty)$ so that for any continuous $f:[0,\infty)\to \RR$, we have that
    \[
        \lim_{X\to\infty}\int_0^1 f\bigg(X^{-\frac 12 }\bigg|\sum_{n\le X} \lambda_f(n)e(n\alpha)\bigg|\bigg)d\alpha = \int fd\mu.
    \]
\end{corollary}
If we restrict the second part of the main theorem to the case $s = 2r$ for $r$ a positive integer, we also obtain the following by orthogonality:
\begin{corollary}
    For all positive integers $r$, we have that 
    \[
        \sum_{\substack{n_1 + \dots + n_r = m_1 + \dots + m_r\\ n_r,\dots,n_r,m_1,\dots,m_r\le X}} 
        \lambda_f(n_1)\dots\lambda_f(n_r)\conj{\lambda_f(m_1)\dots\lambda_f(m_r)} = C_f^{2r}X^r + O(X^{r - \eta_2}).
    \]
    for some constant $C_f^{2r} > 0$.
\end{corollary}

Our proof of the main theorem proceeds via an iterative method, which we sketch below. For $Q\asymp X^{\frac{1}{2} + \delta}$, the $L^s$ integral may 
be approximated by
\[
    \frac{1}{Q^2}\sum_{q\sim Q}\sumCp_{a(q)} \bigg|S_\star\left(\frac{a}{q}; X\right)\bigg|^s
\]
up to some constant factor (the range of $s$ where this works depends on the choice of $\star$). In our case, we achieve this via a version of Jutila's variant of the circle method in \cite{J1}.
Applying Voronoi summation, this roughly reduces to dealing with
\[
    \frac{1}{Q^2}\sum_{q\sim Q}\sumCp_{a(q)} \bigg|S_\star\left(\frac{a}{q}; \frac{q^2}{X}\right)\bigg|^s.
\]

Since $q$ is now much larger than the length of the sum, the inner sum amounts 
to integration over $[0, 1]$, and can be shown to be roughly
\[
    \varphi(q)\int_0^1\bigg|S_\star\prn{\alpha; \frac{q^2}{X}}\bigg|^sd\alpha.
\]
Ignoring the factor of $\varphi(q)$ for purpose of this discussion, note the inner sum varies very little as $q$ varies by small amounts, and so one obtains that
\[
    \frac{1}{Q^2}\sum_{q\sim Q}\sumCp_{a(q)} \bigg|S_\star\left(\frac{a}{q}; \frac{q^2}{X}\right)\bigg|^s\approx
    \frac{c_1}{Q}\int_{\tilde q\asymp Q} \int_0^1 \bigg|S_\star\prn{\alpha; \frac{\tilde q^2}{X}}|^sd\alpha d\tilde q
\]
for some $c_1 > 0$. The same method applied starting with sums of length $X'\asymp X$ with $Q' = \sqrt{X'/X}Q$ yields the same quantity times $(X'/X)^{\frac s2}$. 
The following approximate functional equation is obtained:
\begin{proposition}
    For $X\asymp X', 0 < s < 2, \star\in\set{d, f}$, we have that for some $\eta_1 > 0$
    \begin{equation}
        \int_0^1\bigg|\frac{1}{\sqrt{X}}S_\star(\alpha; X)\bigg|^sd\alpha = \int_0^1\bigg|\frac{1}{\sqrt{X'}}S_\star(\alpha; X')\bigg|^sd\alpha + O(X^{-\eta_1s(2 - s)}).
        \label{eq:app_fe_low_moment}
    \end{equation}
    Furthermore, for $s\ge 2$, for some $\eta_2 > 0$, we have that
    \begin{equation}
        \int_0^1\bigg|\frac{1}{\sqrt{X}}S_f(\alpha; X)\bigg|^sd\alpha = \int_0^1\bigg|\frac{1}{\sqrt{X'}}S_f(\alpha; X')\bigg|^sd\alpha + O(X^{-\eta_2}).
        \label{eq:app_fe_high_moment}
    \end{equation}
    \label{prop:app_fe}
\end{proposition}
Iterating this yields the main theorem. The details of how this implies the main theorem are shown in Section \ref{sec:proof_mainthm}, 
and the proof of Proposition \ref{prop:app_fe} in Section \ref{sec:app_fe_pf}. 
The next section is devoted to a few technical results that are used in the proof.

We remark that our result is related to a result of Jurkat and van Horne \cite{JH} in which 
a limiting distribution is found for magnitude of the exponential sum 
\[
    \bigg|\sum_{n\le X} e(n^2\alpha)\bigg|.
\]
Jurkat and van Horne's methods appear to be quite different from ours, though both 
our result and their result involve applying a Farey dissection and using Poisson
summation (Voronoi summation in our case) on what remains. 

We expect that using our methods applied to the exponential sum with coefficients 
equal to $1*\chi_4$, with $\chi_4$ the character of conductor $4$, we can
strengthen Theorem 4 of \cite{JH}. Specifically, one has that for $0 < s < 4$ and 
some $\delta > 0$:
\begin{equation}
    \int_0^1 \bigg|\sum_{n\le X} e(n^2\alpha)\bigg|^sd\alpha 
    = c_s X^{s/2} + O(X^{\frac{s}{2}(1 - (4 - s)\delta)}),
    \label{eq:JH_strengthen}
\end{equation}
with $c_\alpha$ as in Theorem 4 of \cite{JH}. 
We expect our methods also apply to the case of exponential sums with coefficients 
the Fourier coefficients of Maass forms and half integral weight forms, as we use nothing 
about $f$ besides its modularity via Voronoi summation. 
Such improvements and generalizations should be straightforward and we leave the details to the interested reader.
 
\subsection{Notation and conventions}

As usual, we use Vinogradov's notation $A\ll B$ (equivalently $B\gg A$) to denote that $|A|\le CB$ for some constant $C > 0$. When we use $\eps$ in a statement, we mean that the statement holds for all $\eps > 0$. 
For the purposes of this paper, this $C$ will depend only on $f, s, \eps$, unless specified otherwise.
Any further dependencies will be specified in subscript beneath the $\ll$.
We write $A\asymp B$ to denote that $A\ll B, B\ll A$. In addition, we write $a\sim A$ to denote $A < a\le 2A$. We write
\[\sumCp_{a(q)}\]
to denote a sum over $0\le a < q$ with $(a, q) = 1$.
For convenience, for $\star\in\set{d, f}$, we write 
\[\lambda_\star(n) = \begin{cases} d(n) & \star = d\\ \lambda_f(n) & \star = f\end{cases}.\]

\section{Standard technical lemmas}
We shall use Voronoi summation as stated below, along with some properties of the integral transforms involved. 
These are well-known, and the final bounds follow from repeated integration by parts and trivial bounds 
(see \S2 of \cite{FGKM}, for example).
\begin{proposition}
    Let $w$ be smooth and supported on positive reals, $q\ge 1$ be prime, and $(a, q) = 1$. Then, we have
    \begin{align*}
        \sum_{n\ge 1} d(n)e\pfrc{an}{q}&w\pfrc{n}{X} = \frac{1}{q}\int (\log(x/q^2) + 2\gamma)w\pfrc{n}{X}dx \\
        &+ \frac{X}{q}\sum_{n\ge 1} d(n)e\prn{-\frac{\conj an}{q}}\mc I_d w\pfrc{n}{q^2/X},
    \end{align*}
    \[\sum_{n\ge 1} \lambda_f(n)e\pfrc{an}{q}w\pfrc{n}{X} = \frac{X}{q}\sum_{n\ge 1} \lambda_f(n)e\prn{-\frac{\conj an}{q}}\mc I_f w\pfrc{n}{q^2/X}\]
    where for $\star\in\set{d, f},$
    \[
        \mc I_d w(x) = \int w(x)(4K_0(4\pi\sqrt{x}) - 2\pi Y_0(4\pi\sqrt{x}))dx,
    \]
    \[
        \mc I_f w(x) = \int w(x) J_{k - 1}(4\pi\sqrt{x}) dx.
    \]
    These transforms also satisfy the property that if $w$ is supported on values $\asymp 1$ and $w^{(j)}\ll H^j$ for $j\ge 1$, then 
    $\mc I_\star w\ll 1$, and 
    $(\mc I_\star w)^{(j)}\ll H^j$. Also, for $\delta > 0, |x|\ge H^{1 + \delta}$, we have that 
    $w^{(j)}(x)\ll_{\delta, A} H^{-A}$.
    \label{prop:voronoi_summation}
\end{proposition}

We shall also require and prove a modified version of Jutila's circle method \cite{J1}. It slightly improves the error term 
of Jutila's result slightly in some cases, at the cost of requiring a smoothing.
\begin{proposition}
    Let $\mc Q$ be a set of integers $\asymp Q$. Let $\Delta = \frac{H}{Q^2}$ for some $H\gg 1$. Also, suppose that $\phi$ is some nonzero smooth compactly supported function on $\RR$. Write
    \[L = \sum_{q\in\mc Q}\varphi(q),\]
    \[\tilde\chi(\alpha) = \frac{1}{\hat\phi(0)\Delta L}\sum_{q\in\mc Q}\sumCp_{a(q)}\phi(\Delta^{-1}(\alpha - a/q)).\]
    Then, we have that 
    \[\int_0^1 |1 - \tilde\chi(\alpha)|^2d\alpha\ll \frac{Q^4}{HL^2} + \frac{Q^{2 + \eps}}{L^2}.\]
    \label{prop:jut_circle_method}
\end{proposition}
\begin{proof}
    By Poisson summation, for $(a, q) = 1$ we have
    \[\phi(\Delta^{-1}(\alpha - a/q)) = \Delta\hat\phi(0) + \Delta\sum_{|\ell| > 0} \hat\phi(\Delta\ell) e\left(\frac{a\ell}{q}\right)e(-\ell\alpha),\]
    so
    \begin{align*}
        \tilde\chi(\alpha) - 1 &= \frac{1}{\hat\phi(0)L}\sum_{q\in \mc Q}\sumCp_{a(q)}\sum_{|\ell| > 0} \hat\phi(\Delta\ell) e\left(\frac{a\ell}{q}\right)e(-\ell\alpha)\\
        &= \frac{1}{\hat\phi(0)L}\sum_{|\ell| > 0}\hat\phi(\Delta\ell)e(-\ell\alpha)\sum_{q\in \mc Q} c_q(\ell)
    \end{align*}
    where 
    \[c_q(\ell) = \sumCp_{a(q)} e\left(\frac{a\ell}q\right)\]
    is the usual Ramanujan sum. Using the fact that $c_q(\ell) = \sum_{d|(q, \ell)}\mu(q/d)$, we obtain that
    \[
        \tilde\chi(\alpha) - 1 = \frac{1}{L}\sum_{|\ell| > 0}\hat\phi(\Delta\ell) e(\ell\alpha)\sum_{d | \ell}\sum_{dq_1\in \mc Q}\mu(q_1).
    \]
    By Plancherel, and the bound $\hat\phi(t)\ll_A (1 + |t|)^{-A}$ which holds for all $A > 0$, it follows that for some sufficiently large $C$
    \begin{align*}
        \int_0^1 |\tilde\chi(\alpha) - 1|^2d\alpha &= \frac{1}{\hat\phi(0)^2L^2}\sum_{|\ell| > 0} |\hat\phi(\Delta\ell)|^2\bigg(\sum_{d|\ell}\sum_{dq_1\in \mc Q}\mu(q_1)\bigg)^2\\
        &\ll \frac{1}{L^2}\sum_{|\ell| > 0} |\hat\phi(\Delta\ell)|^2\bigg(\sum_{\substack{d|\ell\\ d\ll Q}}\frac{Q}{d}\bigg)^2\\
        &\ll \frac{1}{L^2}\sum_{K} \frac{1}{K^{10}}\sum_{|\ell|\le K\Delta^{-1}}\bigg(\sum_{\substack{d|\ell\\ d\ll Q}}\frac{Q}{d}\bigg)^2.
    \end{align*}
    where $K$ runs over powers of two. 
    Note that
    \begin{align*}
        \sum_{|\ell|\le K\Delta^{-1}} \bigg(\sum_{\substack{d|\ell\\ d\ll Q}}\frac{Q}{d}\bigg)^2 &= \sum_{d_1, d_2\ll Q}\frac{Q^2}{d_1d_2}\sum_{\substack{\ell\le K\Delta^{-1}\\ [d_1, d_2] | \ell}} 1\\
        &\le K\sum_{d_1, d_2\ll Q}\frac{Q^2}{d_1d_2}\cdot\bigg(\frac{\Delta^{-1}}{[d_1, d_2]} + 1\bigg)\\ 
        &\ll KQ^2\Delta^{-1}\sum_{d_1, d_2\ll Q}\frac{(d_1, d_2)}{d_1^2d_2^2} + KQ^{2 + \eps}\\ 
        &\ll KQ^2\Delta^{-1}\sum_{d_1, d_2\ll Q}\frac{1}{d_1^2d_2^2}\sum_{a|(d_1, d_2)}\varphi(a) + KQ^{2 + \eps}\\ 
        &\ll KQ^2\Delta^{-1}\sum_{a}\frac{\phi(a)}{a^3}\sum_{d_1',d_2'}\frac{1}{d_1'^2}\frac{1}{d_2'^2} + KQ^{2 + \eps}\\
        &\ll K(Q^2\Delta^{-1} + Q^{2 + \eps}).
    \end{align*}
    The desired result follows upon summing over $K$.
\end{proof}

\section{Proof of the main theorem}
\label{sec:proof_mainthm}

In this section, we prove the main theorem assuming Proposition \ref{prop:app_fe}, and in the following section, we prove 
Proposition \ref{prop:app_fe}.
Iterating Proposition \ref{prop:app_fe}, we obtain that for all $0 < s < 2, Y\ge X$ 
\begin{equation}
    \int_0^1\bigg|\frac{1}{\sqrt{X}}S_\star(\alpha; X)\bigg|^sd\alpha = \int_0^1\bigg|\frac{1}{\sqrt{Y}}S_\star(\alpha; Y)\bigg|^sd\alpha
    + O(X^{-\eta_1 s(2 - s)})
    \label{eq:cauchy_low_moment}
\end{equation}
and that for $s \ge 2$
\begin{equation}
    \int_0^1\bigg|\frac{1}{\sqrt{X}}S_f(\alpha; X)\bigg|^sd\alpha = \int_0^1\bigg|\frac{1}{\sqrt{Y}}S_f(\alpha; Y)\bigg|^sd\alpha
    + O(X^{-\eta_2}).
    \label{eq:cauchy_high_moment}
\end{equation}
In particular, the sequence 
\[
    X\mapsto\int_0^1 \bigg|\frac{1}{\sqrt{X}}\sum_{n\le X}\lambda_\star(n)e(n\alpha)\bigg|^sd\alpha
\]
is a Cauchy sequence for all $s\ge 2$ when $\star = f$ and for $0 < s < 2$ 
for general $\star\in\set{d, f}$. 

Taking the limit as $Y\to\infty$ in (\ref{eq:cauchy_low_moment}) and
(\ref{eq:cauchy_high_moment}), we have that for $0 < s < 2,\star\in\set{d, f}$
\[
    \int_0^1\bigg|\frac{1}{\sqrt{X}}S_\star(\alpha; X)\bigg|^sd\alpha = C_s^\star + O(X^{-\eta_1 s(2 - s)}).
\]
for some constants $C_s^\star, \eta_1 > 0$, and for $s\ge 2$
\[
    \int_0^1\bigg|\frac{1}{\sqrt{X}}S_f(\alpha; X)\bigg|^sd\alpha = C_s^f + O(X^{-\eta_2})
\]
for some constants $C_s^f, \eta_2 > 0$.

Thus, Theorem \ref{thm:main_thm} follows if we can show that 
\begin{align}
    \int_0^1\bigg|\frac{1}{\sqrt{X}}S_\star(\alpha; X)\bigg|^sd\alpha\gg 1 && (s > 0)
    \label{eq:lower_bound_all}
\end{align}

(\ref{eq:lower_bound_all}) follows when $\star = f$ from H\"older with the bounds $S_f(\alpha; X)\ll\sqrt{X}$ (Proposition \ref{prop:jut_bound}) and (\ref{eq:rankin_selberg}).
However, one does not have such bounds for $S_d$, so the rest of this section is dedicated to the case of $\star = d$.

In \cite{GP}, it was shown that 
\begin{equation}
    \int_0^1|S_d(\alpha; X)|d\alpha\gg\sqrt{X}.
    \label{eq:lower_bound_d_L1}
\end{equation}
It follows from H\"older that (\ref{eq:lower_bound_all}) holds for $s\ge 1$. Thus, it remains to show:
\begin{proposition}
    We have
    \[\int_0^1|S_d(\alpha; X)|^sd\alpha\gg X^{\frac{s}{2}}.\]
    for $s < 1$.
    \label{prop:lower_bd_d_low}
\end{proposition}
\begin{proof}
    The proof of this follows from Voronoi summation along with the large sieve to deal with the error terms introduced.

    Let $w$ be some smooth function satisfying $\charf{[X^{-\frac{1}{10}}, 1 - X^{-\frac{1}{10}}]}\le w\le\charf{[1/X, 1]}$
    with $w^{(j)}(x) \ll_j X^{\frac{1}{10}j}$ for all $j$. We may then smooth $S_d(\alpha; X)$ by replacing it with
    \[\tilde S_d(\alpha; X) = \sum_{n} d(n)e(n\alpha)w\pfrc{n}{X}\]
    since by Parseval and Cauchy-Schwarz
    \begin{multline*}
        \bigg|\int_0^1 |S_d(\alpha; X)|^s - |\tilde S_d(\alpha; X)|^sd\alpha\bigg|\\
        \le\int_0^1 \bigg|\sum_{n\in [1, X^{\frac{9}{10}}]\cup [X - X^{\frac{9}{10}}, X]}d(n)e(n\alpha)\bigg(1 - w\pfrc{n}{X}\bigg)\bigg|^sd\alpha\\
        \le \bigg(\sum_{n\in [1, X^{\frac{9}{10}}]\cup [X - X^{\frac{9}{10}}, X]} d(n)^2\bigg)^{\frac{s}{2}}\ll X^{\frac{9s}{20} + \eps},
    \end{multline*}
    which is an acceptable error. It thus remains to show that 
    \[\int_0^1 |\tilde S_d(\alpha; X)|^sd\alpha\gg X^{\frac{s}{2}}.\]

    Take $c$ to be some sufficiently small constant. Then, we have that
    \[
        \int_0^1|S_d(\alpha; X)|^sd\alpha
        \ge\sum_{q\sim c\sqrt{X}}\sumCp_{a(q)}
        \int_{-\frac{c}{X}}^{\frac{c}{X}}\bigg|\tilde S_d\bigg(\frac{a}{q} + \beta; X\bigg)\bigg|^sd\beta.
    \]
    Now, note that by Proposition \ref{prop:voronoi_summation} (Voronoi summation), we have that
    \begin{align*}
        S_d\left(\frac{a}{q} + \beta; X\right) = \frac{1}{q}&\int (\log(x/q^2) + 2\gamma - 1)e(x\beta)w\pfrc{x}{X}dx \\
        &+ \frac{1}{q}\sum_{n}d(n)e\left(-\frac{\conj an}{q}\right)\mc I_dw_\beta\pfrc{n}{q^2/X}.\numberthis
        \label{eq:init_voronoi}
    \end{align*}
    where 
    \[
        w_\beta(x) = w(x)e(xX\beta).
    \]
    If $c$ is sufficiently small, then we have the bound $\Re(e(x\beta))\ge 1 - c$ for $0\le x\le X, |\beta|\le\frac{c}{X}$. 
    Therefore, for $q\sim c\sqrt{X}$, we have that
    \begin{align*}
        \bigg|\int &(\log(x/q^2) + 2\gamma - 1)e(x\beta)w\pfrc{x}{X}dx\bigg|\\
        &\ge (1 - c)\int_{q^2}^{X - X^{\frac{9}{10}}} (\log(x/q^2) + 2\gamma - 1)dx - \int_1^{q^2} (\log(q^2/x) + 2\gamma - 1) dx\\
        &\hspace{1cm}\ge (1 + O(c))X.
    \end{align*}
    We remark that for $|\beta|\le\frac{c}{X}$ (which is so in our case), $w_\beta$ satisfies the bounds 
    \[w_\beta^{(j)}(x)\ll X^{\frac{1}{10}j}\]
    and also for $1/X\le x\le 1$
    \begin{equation}
        w_\beta'(x)\ll
        \begin{cases} 
            X^{\frac{1}{10}} & \frac{1}{X}\le x\le X^{-\frac{1}{10}}\\ 
            c & X^{-\frac{1}{10}}\le x\le 1 - X^{-\frac{1}{10}} \\
            X^{\frac{1}{10}} & 1 - X^{-\frac{1}{10}}\le x\le 1.
        \end{cases}
        \label{eq:w_deriv}
    \end{equation}
    Note that the bounds on $\mc I_dw_\beta$ given by Proposition \ref{prop:voronoi_summation}
    imply that the contribution of terms $n\ge X^{\frac{1}{9}}$ is $\ll_A X^{-A}$. 
    It follows from (\ref{eq:init_voronoi}) that
    \[\bigg|\tilde S_d\left(\frac{a}{q} + \beta; X\right)\bigg|^s\ge (1 + O(c))\frac{X^s}{q^s} 
    - \frac{1}{q^s}E(q, a; \beta)^s + O(X^{-2020})\]
    where 
    \[E(q, a; \beta) = 
    \bigg|\sum_{n\le X^{\frac{1}{9}}}d(n)e\left(-\frac{\conj an}{q}\right)\mc I_dw_\beta\pfrc{n}{q^2/X}\bigg|.\]
    By H\"older and the large sieve, we have that 
    \begin{align*}
        \sumCp_{a(q)} E(q, a; \beta)^s&\le\varphi(q)^{1 - \frac{s}{2}}\bigg(\sumCp_{a(q)} E(q, a;\beta)^2\bigg)^{\frac{s}{2}}\\
        &\ll \varphi(q)
        \bigg(\sum_{n\le X^{\frac{1}{9}}}d(n)^2\bigg|\mc I_dw_\beta\pfrc{n}{q^2/X}\bigg|^2\bigg)^{\frac{s}{2}}.
    \end{align*}
    Now, integrating by parts once yields that
    \begin{align*}
        \mc I_dw_\beta\pfrc{n}{q^2/X} &= \int w_\beta\pfrc{x}{X} B_0\pfrc{4\pi\sqrt{nx}}{q}dx \\
        &=\frac{q}{2\pi\sqrt{n}}\int \frac{1}{X}w_\beta'\pfrc{x}{X} B_1\pfrc{4\pi\sqrt{nx}}{q}dx.
    \end{align*}
    We have the standard bounds (see section 8.451 in \cite{GRZ}, for example)
    \begin{align*}
        B_1(x)\ll x^{-\frac{1}{2}} && (x\gg 1),\\
        B_1(x)\ll x^{-1} && (x\ll 1),
    \end{align*}
    where we write $B_\nu = 4(-1)^{\nu}K_\nu - 2\pi Y_\nu$, so it follows that
    \begin{multline*}
        \int w_\beta'\pfrc{x}{X} B_1\pfrc{4\pi\sqrt{nx}}{q}dx
        \ll\frac{q}{\sqrt{n}}\int_1^{\frac{q^2}{n}} \frac{1}{\sqrt{x}}\bigg|\frac{1}{X}w_\beta'\pfrc{x}{X}\bigg| dx \\
        + \frac{q}{n^{\frac{1}{4}}}\int_{q^2/n}^{X - X^{\frac{9}{10}}} \frac{1}{x^{\frac{1}{4}}}\bigg|\frac{1}{X}w_\beta'\pfrc{x}{X}\bigg| dx
        + \frac{q}{n^{\frac{1}{4}}}\int_{X - X^{\frac{9}{10}}}^X \frac{1}{x^{\frac{1}{4}}}\bigg|\frac{1}{X}w_\beta'\pfrc{x}{X}\bigg| dx
        \ll\frac{q}{n^{\frac{1}{4}}}.
    \end{multline*}
    It follows that we have the bound
    \[\int w_\beta\pfrc{x}{X} B_0\pfrc{4\pi\sqrt{nx}}{q}dx\ll\frac{q^2}{n^{\frac{3}{4}}}.\]
    Therefore, it follows that
    \[\sumCp_{a(q)} E(q, a; \beta)^s\ll\varphi(q)\bigg(\frac{q}{\varphi(q)}\bigg)^{\frac{s}{2}}
    \bigg(q^2\sum_{n\le X^{\frac{1}{9}}} d(n)^2n^{-\frac{3}{4}}\bigg)^{\frac{s}{2}}\ll \varphi(q)q^{2s}\cdot\frac{q}{\varphi(q)}.\]
    It follows that so long as $c$ is sufficiently small, for $q\sim c\sqrt{X}$
    \begin{align*}
        \sumCp_{a(q)} \bigg|\tilde S_d\left(\frac{a}{q} + \beta; X\right)\bigg|^s
        &\ge (1 + O(c))\varphi(q)\frac{X^s}{q^s} + O\bigg(\varphi(q)q^s\frac{q}{\varphi(q)}\bigg)\\ 
        &\ge \varphi(q)X^{\frac{s}{2}}\cdot
        \bigg(\frac{1}{2}(2c)^{-s} + O\bigg(c^s\frac{q}{\varphi(q)}\bigg)\bigg).
    \end{align*}
    We thus obtain that
    \[\int_0^1 \bigg|\tilde S_d\left(\frac{a}{q} + \beta; X\right)\bigg|^s d\alpha\ge X^{\frac{s}{2}}
    \int_{-\frac{c}{X}}^{\frac{c}{X}}\sum_{q\sim c\sqrt{X}}\varphi(q)\frac{1}{2}(2c)^{-s} + O(c^s q) d\beta + O(X^{-2020}).\]
    Since 
    \[\sum_{q\sim c\sqrt{X}}\varphi(q) = (1 + o(1))\frac{9}{\pi^2}c^2X,\]
    we obtain 
    \begin{align*}\int_{-\frac{c}{X}}^{\frac{c}{X}}\sum_{q\sim c\sqrt{X}}\varphi(q)\frac{1}{2}(2c)^{-s} + O(c^s q) d\beta
        \ge (1 + o(1))2c\frac{9}{\pi^2}c^2\bigg(\frac{1}{2}(2c)^{-s} + O(c^s)\bigg).
    \end{align*}
    This is $\gg_{c, s} 1$ for $c$ sufficiently small, so the desired result follows.
\end{proof}

\section{Proof of Proposition \ref{prop:app_fe}}
\label{sec:app_fe_pf}
Instead of showing Proposition \ref{prop:app_fe}, we show
Proposition \ref{prop:pre_app_fe}, from 
which Proposition \ref{prop:app_fe} clearly follows.
\addtocounter{equation}{1}
\begin{proposition} 
    \label{prop:pre_app_fe}

    Let $\delta = \frac{1}{100}$.
    Suppose that $X'\asymp X$, and that $X_1\asymp X^{2\delta}$. 
    Also, suppose that $\kappa > 0$ is sufficiently small. 
    Then, for $0 < s < 2$
    \begin{multline*}
        \int_0^1\bigg|\frac{1}{\sqrt{X'}}\sum_{n\le X'}\lambda_\star(n)e(n\alpha)\bigg|^sd\alpha \\ 
        = \frac 23\int_1^2t\int_0^1\bigg|\frac{1}{t\sqrt{X_1}}\sum_{n}\lambda_\star(n)e(n\alpha){\mc I}_\star w\pfrc{n}{X_1t^2}\bigg|^s d\alpha dt + O(X^{-\kappa s(2 - s)}),
        \numberthis
        \label{eq:moment_pre_low}
    \end{multline*}
    and for $s\ge 2$, we have that 
    \begin{multline*}
        \int_0^1\bigg|\frac{1}{\sqrt{X'}}\sum_{n\le X'}\lambda_f(n)e(n\alpha)\bigg|^sd\alpha \\ 
        = \frac 23\int_1^2t\int_0^1\bigg|\frac{1}{t\sqrt{X_1}}\sum_{n}\lambda_f(n)e(n\alpha){\mc I}_f w\pfrc{n}{X_1t^2}\bigg|^s d\alpha dt + O(X^{-\kappa}).
        \numberthis
        \label{eq:moment_pre_high}
    \end{multline*}
\end{proposition}

We now proceed to prove Proposition \ref{prop:pre_app_fe} in the remainder of this section. After some initial setup, we 
shall show (\ref{eq:moment_pre_low}), (\ref{eq:moment_pre_high}) separately. 
The proofs of the two are largely similar, with (\ref{eq:moment_pre_high}) being
slightly simpler.

Let $ \delta_1 = \delta^{100}$.
Let $w$ be some smooth function satisfying $\charf{[X^{-\delta_1}, 1 - X^{-\delta_1}]}\le w\le\charf{[C/X, 1]}$ for some 
$C > 0$ with $w^{(j)}(x) \ll_j X^{\delta_1 j}$ for all $j$. Then, let
\[
    I_\star^s = \int_0^1 \bigg|\frac{1}{\sqrt{X'}}\sum_{n} \lambda_\star(n)e(n\alpha)w\pfrc{n}{X'}\bigg|^sd\alpha.
\]
We shall use the following lemma to show that working with this smoothed exponential sum results in an acceptable loss.
\begin{lemma}
    We have that for $0 < s < 2$
    \begin{equation}
        I_\star^s-\int_0^1\bigg|\frac{1}{\sqrt{X'}}\sum_{n\le X'}\lambda_\star(n)e(n\alpha)\bigg|^sd\alpha\ll X^{-\frac{\delta_1}{4}s + \eps}.
        \label{eq:smooth_error_d}
    \end{equation}
    Furthermore, for $s\ge 2$
    \begin{equation}
        I_f^s-\int_0^1\bigg|\frac{1}{\sqrt{X'}}\sum_{n\le X'}\lambda_f(n)e(n\alpha)\bigg|^sd\alpha\ll X^{-\delta_1 + \eps}.
        \label{eq:smooth_error_f}
    \end{equation}
\end{lemma}
\begin{proof}
    For $s\le 1$, we have the inequality $||x|^s - |y|^s|\le |x - y|^s$ for any  $x, y\in\CC$ so it follows by Cauchy-Schwarz 
    and the bound $|\lambda_\star(n)|\le d(n)$ that
    \begin{align*}
        I_\star^s - \int_0^1& \bigg|\frac{1}{\sqrt{X'}}\sum_{n\le X'} \lambda_\star(n)e(n\alpha)\bigg|^sd\alpha\\ 
        &\ll\int_0^1\bigg|\frac{1}{\sqrt{X'}}
        \sum_{n\in [0, X'X^{-\delta_1}]\cup [X' - X'X^{-\delta_1}, X']}\lambda_\star(n)e(n\alpha)\bigg(1 - w\pfrc{n}{X'}\bigg) 
        \bigg|^sd\alpha\\
        &\hspace{2cm}\ll X^{-\frac{\delta_1}{2}s + \eps}.
        \numberthis
        \label{eq:smooth_error_low}
    \end{align*}
    Now, for $1 <  s\le 2$, by H\"older and Plancherel, we have
    \[I_\star^{s - 1}, \int_0^1\bigg|\frac{1}{\sqrt{X'}}\sum_{n\le X'} \lambda_\star(n)e(n\alpha)\bigg|^{s - 1}d\alpha\ll X^\eps.\]
    For $1 < s$, we also have that $||x|^s - |y|^s|\ll (|x|^{s - 1} + |y|^{s - 1})|x - y|$, so for $1 < s < 2$, we have
    \begin{align*}
        I_\star^s& - \int_0^1 \bigg|\frac{1}{\sqrt{X'}}\sum_{n\le X'} \lambda_\star(n)e(n\alpha)\bigg|^sd\alpha\\
        &\ll X^{\eps}\int_0^1\bigg|\frac{1}{\sqrt{X'}}
        \sum_{n\in [0, X'X^{-\delta_1}]\cup [X' - X'X^{-\delta_1}, X']}\lambda_\star(n)e(n\alpha)\bigg(1 - w\pfrc{n}{X'}\bigg) 
        \bigg|d\alpha\\
        &\ll X^{-\frac{\delta_1}{2} + \eps}.\numberthis
        \label{eq:smooth_error_high}
    \end{align*}
    Combining (\ref{eq:smooth_error_low}), (\ref{eq:smooth_error_high}), we obtain (\ref{eq:smooth_error_d}).

    To show (\ref{eq:smooth_error_f}), note that we have that by Proposition \ref{prop:jut_bound} and partial summation
    \[\frac{1}{\sqrt{X'}}\sum_{n}\lambda_f(n)e(n\alpha)w\pfrc{n}{X'}, \frac{1}{\sqrt{X'}}\sum_{n\le X'}\lambda_f(n)e(n\alpha)\ll 1.\]
    Then, for $s\ge 2$ we obtain from Cauchy-Schwarz and Plancherel that
    \begin{align*}
        I_f^s &- \int_0^1 \bigg|\frac{1}{\sqrt{X'}}\sum_{n\le X'} \lambda_\star(n)e(n\alpha)\bigg|^sd\alpha\\
        &\ll \frac{1}{\sqrt{X'}}\int_{0}^1\bigg|\sum_{n\in [0, X'X^{-\delta_1}]\cup [X' - X'X^{-\delta_1}, X']}\lambda_\star(n)e(n\alpha)\bigg(1 - w\pfrc{n}{X'}\bigg)\bigg|^sd\alpha\\
        &\ll X^{-\frac{\delta_1}{2} + \eps}
    \end{align*}
    so (\ref{eq:smooth_error_f}) follows from Cauchy-Schwarz and Plancherel.
\end{proof}
\subsection{Proof of (\ref{eq:moment_pre_low})}

At this point we require the following estimate.
\begin{lemma}
    Suppose that 
    \[
        Q\asymp X^{\frac{1}{2} + \delta}, L = \sum_{q\sim Q}\varphi(q).
    \]
    Then, we have that for $0 < s < 2$
    \begin{equation}
        I_\star^s = \frac{1}{L}\sum_{q\sim Q}\sumCp_{a(q)} \bigg|\sum_n \lambda_\star(n)e\pfrc{an}{q}w\pfrc{n}{X}\bigg|^s 
        + O(X^{-\frac{\delta}{4}s(2 - s) + \eps}).
        \label{eq:jut_nophase}
    \end{equation}
    \label{lem:jut_nophase}
\end{lemma}
\begin{proof}

    Let $\phi$ be some nonzero smooth function supported on $[1, 2]$ with $\int\phi = 1$, and let
    \[
        H = X^\delta, \Delta = \frac{H}{Q^2},
    \]
    \[
        \tilde\chi(\alpha) = \frac{1}{\Delta L}\sum_{q\sim Q}\sumCp_{a(q)}\phi(\Delta^{-1}(\alpha - a/q)).
    \]

    Since the fractions  $\set{\frac{a}{q} : q\in\mc Q, (a, q) = 1}$ are at least $Q^{-2}$-spaced, we have the bound
    \[\tilde\chi(\alpha)\ll \frac{1}{L\Delta}\cdot (1 + \Delta Q^2)\ll Q^\eps\] 
    as $L\gg Q^{2}$.

    Now, let
    \[
        \tilde I_\star^s = 
        \int_0^1 \bigg|\frac{1}{\sqrt{X'}}\sum_{n}\lambda_\star(n)e(n\alpha)w\pfrc{n}{X'}\bigg|^s
        \tilde\chi(\alpha)d\alpha.
    \]
    By H\"older, we have that
    \[
        |I_\star^s - \tilde I_\star^s|
        \le\bigg(\int_0^1 \bigg|\frac{1}{\sqrt{X'}}S_\star(\alpha; X')\bigg|^2d\alpha\bigg)^{\frac{s}{2}}
        \bigg(\int_0^1 |1 - \tilde\chi(\alpha)|^{\frac{s}{2 - s}}d\alpha\bigg)^{1 - \frac{s}{2}}.
    \]
    Note that for $0 < s < 2$, we have that $\frac{s}{2 - s}\ge\frac{s}{2}$. By Plancherel and our pointwise bound on $\tilde\chi$, it
    follows that the above is
    \[\ll X^\eps\bigg(\int_0^1 |1 - \tilde\chi(\alpha)|^{\frac{s}{2}}d\alpha\bigg)^{1 - \frac{s}{2}}.\]
    By Proposition \ref{prop:jut_circle_method} and H\"older, we may therefore conclude that
    \begin{equation}
        |I_\star^s - \tilde I_\star^s|\ll X^{-\frac{\delta}{4}s(2 - s) + \eps}.
        \label{eq:jut_error_term}
    \end{equation}
    By the definition of $\tilde\chi$, we have that
    \[
        \tilde I_\star^s 
        = \int_1^2 \phi(t)\sum_{q\sim Q}\sumCp_{a(q)}\bigg|\frac{1}{\sqrt{X'}}\sum_{n} \lambda_\star(n)e\pfrc{an}{q}e(nt\Delta)
        w\pfrc{n}{X'}\bigg|^sdt.
    \]
    Now, note that for all $\beta\in [\Delta, 2\Delta]$, we have that
    \begin{multline*}
        \frac{1}{\sqrt{X'}}\sum_{n} \lambda_\star(n)e\pfrc{an}{q}e(n\beta)w\pfrc{n}{X'} \\
        = \frac{1}{\sqrt{X'}}e(X'\beta)\sum_n \lambda_\star(n)e\pfrc{an}{q}w\pfrc{n}{X'}
        + O(X\Delta E(q, a, \beta)),
    \end{multline*}
    where 
    \[
        E(q, a, \beta) = \sup_{I\subset [1, X']}\bigg|\frac{1}{\sqrt{X'}}\sum_{n\in I} \lambda_\star(n)e\pfrc{an}{q}w\pfrc{n}{X'}\bigg|.
    \]
    In general, for positive $A, B$, we have the inequality $(A + O(B))^s = A^s + O(B^s + A^{s - \min(1, s)}B^{\min(1, s)})$. It 
    follows that
    \[
        \bigg|\frac{1}{\sqrt{X'}}\sum_{n} \lambda_\star(n)e\pfrc{an}{q}e(n\beta)w\pfrc{n}{X'}\bigg|^s
        = \bigg|\frac{1}{\sqrt{X'}}\sum_n \lambda_\star(n)e\pfrc{an}{q}w\pfrc{n}{X'}\bigg|^s 
    \]
    plus an error that is 
    \[
        \ll(X\Delta)^sE(q, a, \beta)^s.
    \]
    if $s\le 1$ and 
    \begin{align*}
        \ll (X\Delta)^s&E(q, a, \beta)^s + (X\Delta)E(q, a, \beta)\bigg|\sum_n \lambda_\star(n)e\pfrc{an}{q}w\pfrc{n}{X'}\bigg|^{s - 1}\\
        &\ll (X\Delta)^{\frac{s}{2}}E(q, a, \beta)^s
    \end{align*}
    if $1 < s < 2$.
    Since $(X\Delta)\ll X^{-\delta}$, we have that both of these error terms are 
    \[
        \ll (X\Delta)^{\frac{s}{2}}\bigg(E(q, a, \beta)^s 
        + \bigg|\frac{1}{\sqrt{X'}}\sum_n \lambda_\star(n)e\pfrc{an}{q}w\pfrc{n}{X'}\bigg|^s\bigg).
    \]
    At this point, we shall use a maximal version of the large sieve 
    of Montgomery (see Theorem 2 of \cite{Mo}), which we 
    state below:
    \begin{proposition}[Maximal large sieve, \cite{Mo}]
        For any sequence $a : \NN\to\CC$ supported on an interval $I$ of length
        $N$, and $\eta$-separated frequencies $\alpha_1,\dots,\alpha_R$, we have that 
        \[
            \sum_{r\le R}\sup_{J\subset I}\bigg|\sum_{n\in J} a(n)e(n\alpha_r)\bigg|^2
            \ll(\eta^{-1} + N)\sum_n |a(n)|^2
        \]
        where $J$ ranges over subintervals of $I$
    \end{proposition}
    Then, by H\"older and the maximal large sieve, we have that 
    \begin{equation}
        \frac{1}{\varphi(q)}\sumCp_{a(q)} E(q, a, \beta)^s 
        + \bigg|\frac{1}{\sqrt X'}\sum_n \lambda_\star(n)e\pfrc{an}{q}w\pfrc{n}{X'}\bigg|^s\ll X^\eps.
        \label{eq:small_phase_err_bd}
    \end{equation}
    Combining the above with (\ref{eq:jut_error_term}), we obtain that 
    \begin{equation}
        I_\star^s = \sum_{q\sim Q}\sumCp_{a(q)}\bigg|\frac{1}{\sqrt{X'}}\sum_{n} \lambda_\star(n)e\pfrc{an}{q}w\pfrc{n}{X'}\bigg|^s 
        + O(X^{-\frac{\delta}{4}s(2 - s) + \eps}).
        \label{eq:jut_circle_method_nophase}
    \end{equation}
\end{proof}

Now, take $Q = (X'X_1)^{\frac{1}{2}}$, and suppose that $q\sim Q$.

By Proposition \ref{prop:voronoi_summation}, we have that
(noting that the main term in the case of $\star = d$ can be easily absorbed
into the error term)
\begin{multline*}
    \frac{1}{\sqrt{X'}}\sum_{n} \lambda_\star(n)e\pfrc{an}{q}w\pfrc{n}{X'}\\
    = \frac{1}{\sqrt{X_1(q/Q)^2}}\sum_n\lambda_\star(n)e\prn{-\frac{\conj an}{q}}\mc I_\star w\pfrc{n}{X_1(q/Q)^2} + O(X^{-\frac{\delta s}{2} +\eps}).
\end{multline*}
From the properties of $I_\star w$ noted in Proposition \ref{prop:voronoi_summation}, we have that 
$\mc I_\star w(x)$ is $\ll_A X^{-A}$ for $x\ge X^{2\delta_1}$. 
Furthermore, we have the trivial bound
$\mc I_\star w(x)\ll 1$, and $(\mc I_\star w)^{(j)}(x)\ll X^{\delta_1}$ for $j\ge 1$.

Thus, noting that $a\mapsto -\conj a$ is a permutation of $(\ZZ/q\ZZ)^\times$, combining the above with 
Lemma \ref{lem:jut_nophase} (observing that $Q^2/X' = X_1$)
\begin{multline*}
    I_\star^s = \frac{1}{L}\sum_{q\sim Q}\sumCp_{a(q)}
    \bigg|\frac{1}{\sqrt{X_1(q/Q)^2}}\sum_{n\le X_1X^{2\delta_1}}
    \lambda_\star(n)e\pfrc{an}{q}\mc I_\star w\pfrc{n}{X_1(q/Q)^2}\bigg|^s \\+ O(X^{-\frac{\delta}{4}s(2 - s) + \eps}).
\end{multline*}
We now deal with the inner sum with the following lemma, which amounts to the observation that 
averaging the exponential sum over fractions $\frac{a}{q}$ with $(a, q) = 1, 0 < a < q$ is essentially
integration over $[0, 1]$. This is because these fractions are equidistributed on scales much smaller than the reciprocal of the length of the exponential sum since $q$ is large.

This observation is stated more generally in the following lemma. 
\begin{lemma}
    Suppose that $a : \NN\to\CC$ is a sequence of complex numbers satisfying $|a(n)|\ll n^\eps$.  
    Then, if $q\ge Y^{10}$, $0 < s < 2$, we have that
    \[\frac{1}{\varphi(q)}\sumCp_{a(q)}\bigg|\frac{1}{\sqrt{Y}}\sum_{n\le Y} a(n)e\pfrc{an}{q}\bigg|^s 
    = \int_0^1\bigg|\frac{1}{\sqrt{Y}}\sum_{n\le Y} a(n)e(n\alpha)\bigg|^sd\alpha + O(Y^{-s}).\]
    \label{lem:coprime_riemann_sum}
\end{lemma}
It should be possible to show this for $q\ge Y^{1 + \eps}$ for any $\eps > 0$ 
(and an error of $Y^{-c\eps s}$), though our result suffices.

To prove this lemma, we first prove the following crude result on the equidistribution of reduced residue classes in short intervals. 
\begin{lemma} 
    For $q\ge 1, H\ge 1$, we have that
    \[
        N_q([x, x + H]) = \frac{\varphi(q)}{q}H + O(d(q)).
    \]
    where 
    \[
        N_q(I) = \sum_{\substack{n\in I\\ (n, q) = 1}} 1
    \]
    \label{lem:equidistribute_rred_class}
\end{lemma}
\begin{proof}
    By M\"obius inversion, we have that for any $D\ge 1$
    \begin{align*}
        \sum_{\substack{x\le n\le x + H\\ (n, q) = 1}} 1 &= \sum_{\substack{x\le n\le x + H\\ (n, q) = 1}}\sum_{\substack{d | n\\ d|q}}\mu(d)\\
        &= \sum_{d | q}\mu(d)\sum_{\substack{x\le n\le x + H\\ d | n}} 1\\
        &= \sum_{d | q} \mu(d)\left(\frac{H}{d} + O(1)\right) = H\sum_{d | q}\frac{\mu(d)}{d} + O(d(q)).
    \end{align*}
    The desired result follows upon noting that $\frac{\varphi(q)}{q} = \sum_{d | q}\frac{\mu(d)}{d}$.
\end{proof}
\begin{proof}[Proof of Lemma \ref{lem:coprime_riemann_sum}]
    First, note that for all $|\beta|\le q^{-\frac{1}{2}}$, we have that
    \[\sum_{n\le Y} a(n)e\pfrc{an}{q}e(n\beta) = \sum_{n\le Y} a(n)e\pfrc{an}{q} + O(Y^2|\beta|).\]
    From our bound on $|\beta|$, the error term then must be $\le Y^{-3}$. It follows from Lemma \ref{lem:equidistribute_rred_class} that
    \begin{align*}
        \frac{1}{\varphi(q)}&\sumCp_{a(q)}\bigg|\frac{1}{\sqrt{Y}}\sum_{n\le Y} a(n)e\pfrc{an}{q}\bigg|^s \\
        &= \frac{1}{\varphi(q)}\sumCp_{a(q)}q^{\frac{1}{2}}\int_{0}^{q^{-\frac{1}{2}}}
        \bigg|\frac{1}{\sqrt{Y}}\sum_{n\le Y} a(n)e\pfrc{an}{q}e(n\beta)\bigg|^sd\alpha + O(Y^{-s})\\
        &= \frac{q^{\frac{1}{2}}}{\varphi(q)}\int_0^1 \bigg|\frac{1}{\sqrt{Y}}\sum_{n\le Y}a(n)e(n\alpha)\bigg|^s
        N_q([q\alpha, q\alpha + q^{\frac{1}{2}}])d\alpha + O(Y^{-s})\\
        &= \int_0^1 \bigg|\frac{1}{\sqrt{Y}}\sum_{n\le Y}a(n)e(n\alpha)\bigg|^sd\alpha 
        + O\bigg(q^{-1 + \eps}\int_0^1 \bigg|\frac{1}{\sqrt{Y}}\sum_{n\le Y}a(n)e(n\alpha)\bigg|^sd\alpha\bigg).
    \end{align*}
    By H\"older, the error is at most $q^{-1 + \eps}$, and the desired result follows.
\end{proof}
Applying Lemma \ref{lem:coprime_riemann_sum}, we obtain that for $0 < s < 2$
\begin{align*}
    I_\star^s = \frac{1}{L}\sum_{q\sim Q}\varphi(q)\int_0^1\bigg|\frac{1}{\sqrt{X_1(q/Q)^2}}&\sum_{n\le X_1X^{2\delta_1}}
    \lambda_\star(n)e(n\alpha) \mc I_\star w\pfrc{n}{X_1(q/Q)^2}\bigg|^sd\alpha\\
    &+ O(X^{-\frac{\delta}{4}s(2 - s) + \eps}).
    \numberthis\label{eq:jut_riemann}
\end{align*}
We now show that the sum over $q\sim Q$ may be turned into an integral over $[Q, 2Q]$ at a small loss. We first eliminate the $\varphi(q)$ at the cost of averaging over $q$ in a short interval. 
To see that this is so, consider some $q\sim Q$ and a real $\tilde q\in [Q, 2Q]$. Then, we have that 
\[
    \frac{1}{\sqrt{X_1(q/Q)^2}} = \frac{1}{\sqrt{X_1(\tilde q/Q)^2}} + O\bigg(|q - \tilde q|\frac{1}{Q\sqrt{X_1}}\bigg).
\]
Note that $Q\sqrt{X_1}\gg X_1^{20}$. We also have that from the derivative bounds on $\mc I_\star w$ that for 
$n\le X_1X^{2\delta_1}$
\[
    \mc I_\star w\pfrc{n}{X_1(q/Q)^2} - \mc I_\star w\pfrc{n}{X_1(\tilde q/Q)^2}\ll |q - \tilde q|\cdot\frac{X^{3\delta_1}}{Q}.
\]
Then, for all $\alpha\in\RR$
\begin{align*}
    &\frac{1}{\sqrt{X_1(q/Q)^2}} \sum_{n\le X_1X^{2\delta_1}}\lambda_\star(n)e(n\alpha)\mc I_\star w\pfrc{n}{X_1(q/Q)^2}\\
    &= \frac{1}{\sqrt{X_1(q/Q)^2}} \sum_{n\le X_1X^{2\delta_1}}\lambda_\star(n)e(n\alpha)\mc I_\star w\pfrc{n}{X_1(\tilde q/Q)^2}\\
    &\hspace{2cm} + O\bigg(|q - \tilde q|\sqrt{X_1}X^{5\delta_1}Q^{-1}\bigg)\\
    &= \frac{1}{\sqrt{X_1(\tilde q/Q)^2}} \sum_{n\le X_1X^{2\delta_1}}\lambda_\star(n)e(n\alpha)\mc I_\star w\pfrc{n}{X_1(\tilde q/Q)^2}\\
    &\hspace{2cm} + O\bigg(|q - \tilde q|\sqrt{X_1}X^{5\delta_1}Q^{-1}\bigg).
\end{align*}
If $|q - \tilde q|\ll Q^{\frac{1}{10}}$, note that we must have $|q - \tilde q|\sqrt{X_1}X^{5\delta_1}Q^{-1}\ll X^{-20\delta}$ so
since $\tilde q/q = 1 + O(Q^{-9/10})$ (and since by Plancherel and Ho\"older, the below is $\ll X^\eps$), we have
\begin{align*}
    &\int_0^1 
    \bigg|\frac{1}{\sqrt{X_1(q/Q)^2}}\sum_{n\le X_1X^{2\delta_1}}\lambda_\star(n)e(n\alpha)\mc I_\star w\pfrc{n}{X_1(q/Q)^2}\bigg|^sd\alpha\\
    &= 
    \int_ 0^1
    \bigg|\frac{1}{\sqrt{X_1(\tilde q/Q)^2}}\sum_{n\le X_1X^{2\delta_1}}\lambda_\star(n)e(n\alpha)\mc I_\star w\pfrc{n}{X_1(\tilde q/Q)^2}
    \bigg|^sd\alpha + O(X^{-10\delta s + \eps}).\\ 
    &= \frac{\tilde q}{q}\int_ 0^1
    \bigg|\frac{1}{\sqrt{X_1(\tilde q/Q)^2}}\sum_{n\le X_1X^{2\delta_1}}\lambda_\star(n)e(n\alpha)\mc I_\star w\pfrc{n}{X_1(\tilde q/Q)^2}
    \bigg|^sd\alpha + O(X^{-10\delta s + \eps}).
    \numberthis\label{eq:q_interp}
\end{align*}

Now, let $\set{I_j}_{1\le j\le K}$ be a partition of $[Q, 2Q]$ into intervals of length $\asymp Q^{\frac{1}{10}}$ for some 
$K\asymp Q^{1 - \frac{1}{10}}$. We obtain from (\ref{eq:q_interp}), adding redundant averaging over $\tilde q\sim Q$, that
\begin{align*}
    \frac{1}{L}\sum_{q\sim Q}\varphi(q)&\int_0^1\bigg|\frac{1}{\sqrt{X_1(q/Q)^2}}\sum_{n\le X_1X^{2\delta_1}}
    \lambda_\star(n)e(n\alpha) \mc I_\star w\pfrc{n}{X_1(q/Q)^2}\bigg|^s d\alpha\\
    = \frac{1}{L}\sum_{j = 1}^K&\sum_{q\in I_j}\varphi(q)\int_0^1\bigg|\frac{1}{\sqrt{X_1(q/Q)^2}}\sum_{n\le X_1X^{2\delta_1}}
    \lambda_\star(n)e(n\alpha) \mc I_\star w\pfrc{n}{X_1(q/Q)^2}\bigg|^s d\alpha\\ 
    = \frac{1}{L}&\sum_{j = 1}^K \sum_{q\in I_j}\frac{\varphi(q)}{q} 
    \frac{1}{|I_j|}\int_{I_j} \tilde q\int_0^1 [\dots] d\alpha d\tilde q + O(X^{-10\delta s + \eps})
    \numberthis\label{eq:q_interval_avg}
\end{align*}
where $[\dots]$ is  
\[
\bigg|\frac{1}{\sqrt{X_1(\tilde q/Q)^2}}\sum_{n\le X_1X^{2\delta_1}}\lambda_\star(n)e(n\alpha)
\mc I_\star w\pfrc{n}{X_1(\tilde q/Q)^2}\bigg|^s. 
\]
It follows from well-known elementary results that 
\[L = \frac{9}{\pi^2}Q^2 + O(Q\log Q), \sum_{n\in I_j}\frac{\varphi(q)}{q} = \frac{6}{\pi^2}|I_j| + O(\log Q).\]
Thus (\ref{eq:q_interval_avg}) equals 
\[
    \frac{2}{3Q^2}\int_Q^{2Q}\tilde q\int_0^1\bigg|\frac{1}{\sqrt{X_1(\tilde q/Q)^2}}\sum_{n\le X_1X^{2\delta_1}}
    \lambda_\star(n)e(n\alpha)\mc I_\star w\pfrc{n}{X_1(\tilde q/Q)^2}\bigg|^sd\alpha d\tilde q
\]
plus an error of $Q^{-1 + \eps}\ll X^{-\frac{1}{2}}$. By the change of variables $t = \tilde q/Q$, this equals
\[
    \frac{2}{3}\int_1^2 t\int_0^1\bigg|\frac{1}{\sqrt{X_1t^2}}\sum_{n\le X_1X^{2\delta_1}}
    \lambda_\star(n)e(n\alpha)\mc I_\star w\pfrc{n}{X_1t^2}\bigg|^sd\alpha dt.
\]
Completing the sum over $n$ with the same method by which it was removed, and gathering up 
(\ref{eq:smooth_error_d}), (\ref{eq:jut_riemann}), we obtain that for some $\kappa > 0$
\begin{multline*}
    \int_0^1\bigg|\frac{1}{\sqrt{X'}}\sum_{n\le X'}\lambda_\star(n)e(n\alpha)\bigg|^sd\alpha\\ 
    = \frac{2}{3}\int_1^2 t\int_0^1\bigg|\frac{1}{t\sqrt{X_1}}\sum_{n}\lambda_\star(n)e(n\alpha)\mc I_\star w\pfrc{n}{X_1t^2}\bigg|^s d\alpha dt
    + O(X^{-\kappa s(2 - s)}).
\end{multline*}
for any $X_1\asymp X^{2\delta}$. The desired result follows.
\subsection{Proof of (\ref{eq:moment_pre_high})}

This section proceeds a similar fashion to the previous section, so we shall refer to it heavily and only indicate those 
places in which the two differ. What allows us to deal with higher moments
is the following bound of Jutila (which improves by a factor of $\log X$ on an older bound of Wilton, which we could have also used):
\begin{proposition}[\cite{J}]
    For all $L, R$, we have that
    \[
        \bigg|\sum_{L < n\le R} \lambda_f(n)e(n\alpha)\bigg|\ll R^{\frac{1}{2}}.
    \]
    \label{prop:jut_bound}
\end{proposition}
As a corollary of this and partial summation, it follows that for any $Y_1\le Y$
\[
    \bigg|\sum_{Y_1 < n\le Y} \lambda_f(n)e(n\alpha)w\pfrc{n}{Y}\bigg|\ll Y^{\frac{1}{2}}.
\]

For the rest of the section, we also suppose that $s\ge 2$.
By (\ref{eq:smooth_error_f}), it suffices to show that for $X'\asymp X$
\[
    I_f^s = \frac{2}{3}\int_1^2 t\int_0^1\bigg|\frac{1}{t\sqrt{X_1}}\sum_{n}\lambda_f(n)e(n\alpha)\mc I_f w\pfrc{n}{X_1t^2}\bigg|^s d\alpha dt + O(X^{-\eta})
\]
for some $\eta > 0$, and this is what the rest of the section is devoted to showing. 

We now show the following analogue of (\ref{eq:jut_nophase})

\begin{lemma}
    Suppose that \[Q\asymp X^{\frac{1}{2} + \delta}, L = \sum_{q\sim Q}\varphi(q).\]
    Then, we have that for $s\ge 2$
    \begin{equation}
        I_f^s = \frac{1}{L}\sum_{q\sim Q}\sumCp_{a(q)} \bigg|\sum_n \lambda_f(n)e\pfrc{an}{q}w\pfrc{n}{X}\bigg|^s 
        + O(X^{-\frac{\delta}{4}}).
        \label{eq:jut_nophase_high}
    \end{equation}
\end{lemma}
\begin{proof}
    Let $\phi$ be some nonzero smooth function supported on $[1, 2]$ with $\hat\phi(0) = 1$. 
    Also, as in the previous subsection, let 

    \[
        H = X^\delta, \Delta = \frac{H}{Q^2},
    \] 
    \[
        \tilde\chi(\alpha) = \frac{1}{\Delta L}\sum_{q\sim Q}\sumCp_{a(q)}\phi(\Delta^{-1}(\alpha - a/q)),
    \]
    \[
        \tilde I_\star^s = 
        \int_0^1 \bigg|\frac{1}{\sqrt{X'}}\sum_{n}\lambda_\star(n)e(n\alpha)w\pfrc{n}{X'}\bigg|^s
        \tilde\chi(\alpha)d\alpha.
    \]

    Then, by Proposition \ref{prop:jut_bound} and partial summation, Proposition \ref{prop:jut_circle_method}, and Cauchy Schwarz
    \begin{equation}
        I_f^s -  \tilde I_f^s\ll\int_0^1 |1 - \tilde\chi(\alpha)|d\alpha\ll\bigg(\int_0^1 |1 - \tilde\chi(\alpha)|^2d\alpha\bigg)^{\frac{1}{2}}
        \ll X^{-\frac{\delta}{2}}, 
        \label{eq:jut_high_moment}
    \end{equation}
    so the error from replacing $I_f^s$ with $\tilde I_f^s$ is admissible.
    For $\beta\sim\Delta$, we have that by partial summation and Proposition \ref{prop:jut_bound} 
    \[
        \frac{1}{\sqrt{X}}\sum_n \lambda_f(n)e\pfrc{an}{q}e(n\beta)w\pfrc{n}{X} 
        = \frac{1}{\sqrt{X}}\sum_n \lambda_f(n)e\pfrc{an}{q}w\pfrc{n}{X} + O(X^{-\delta})
    \]
    so 
    \[
        \bigg|\frac{1}{\sqrt{X}}\sum_n \lambda_f(n)e\pfrc{an}{q}e(n\beta)w\pfrc{n}{X}\bigg|^s 
        = \bigg|\frac{1}{\sqrt{X}}\sum_n \lambda_f(n)e\pfrc{an}{q}w\pfrc{n}{X}\bigg|^s + O(X^{-\delta})
    \]
    and the desired result follows from this and (\ref{eq:jut_high_moment}) immediately.
\end{proof}

Take $Q = (X'X_1)^{1/2}$. Applying Voronoi and bounds for tails of $\mc I_f w$ 
as in the previous section, we obtain that
\[
    I_f^s = \frac{1}{L}\sum_{q\sim Q}\sumCp_{a(q)}\bigg|\frac{1}{\sqrt{X_1(q/Q)^2}}\sum_{n\le X_1X^{2\delta_1}}
    \lambda_f(n)e\pfrc{an}{q}\mc I_f w\pfrc{n}{X_1(q/Q)^2}\bigg|^s.
\]
plus an error of $O(X^{-\frac{\delta}{2}})$

We now require the following analogue of Lemma \ref{lem:coprime_riemann_sum} for high moments.
\begin{lemma}
    Suppose that $q\ge (YY')^{20}$. Also, suppose that $w$ is smooth and compactly 
    supported away from $1$ with $\int |w'|\ll 1, |w^{(j)}|\ll Y'^{j}$. Then we have that for 
    $s\ge 1$ and $Y$ sufficiently large
    \begin{align*}
        \frac{1}{\varphi(q)}\sumCp_{a(q)}
        \bigg|\frac{1}{\sqrt{Y}}&\sum_{n\le YY'^2} \lambda_f(n)e\pfrc{an}{q}\mc I_fw\pfrc{w}{Y}\bigg|^s \\
        &= \int_0^1\bigg|\frac{1}{\sqrt{Y}}\sum_{n\le YY'^2} \lambda_f(n)e(n\alpha)\mc I_fw\pfrc{w}{Y}\bigg|^sd\alpha + O( (YY')^{-1} ).
    \end{align*}
    \label{lem:coprime_riemann_sum_high}
\end{lemma}
\begin{proof}
    Proceeding as in the proof of Lemma \ref{lem:coprime_riemann_sum}, we have that
    for $|\beta|\ll q^{-\frac 12}$ 
    \begin{align*}
        \frac{1}{\sqrt{Y}}\sum_{n\le YY'} \lambda_f(n)e\prn{\frac{an}{q}}&
        \mc I_fw\pfrc{n}{Y}\\
        &= \frac{1}{\sqrt{Y}}\sum_{n\le YY'} \lambda_f(n)e\prn{\frac{an}{q} + \beta n}
        \mc I_fw\pfrc{n}{Y} + O((YY')^{-3}). 
    \end{align*}
    Now, using the bounds on $\mc I_fw$ in Proposition \ref{prop:voronoi_summation}, we may 
    complete the sum over $n$ at the cost of an error of $O((YY')^{-100})$. Applying Voronoi
    summation to the complete sum yields that
    \begin{multline*}
        \frac{1}{\sqrt{Y}}\sum_{n\le YY'} \lambda_f(n)e\prn{\frac{an}{q}}
        \mc I_fw\pfrc{n}{Y}\\
        = \frac{1}{\sqrt{Y}}\sum_n \lambda_f(n)e\pfrc{an}{q}\mc I_fw\pfrc{n}{Y} + O( (YY')^{-100}) \\
        = \frac{1}{\sqrt{q^2/Y}}\sum_{n\asymp q^2/Y} \lambda_f(n)e\pfrc{\conj an}{q}w\pfrc{n}{q^2/Y}
        + O( (YY')^{-100}).
    \end{multline*}
    By Jutila's bound (Proposition \ref{prop:jut_bound}), we have that this is $\ll 1$, so 
    it follows that since all $\alpha$ are of the form $\frac{a}{q} + \beta$ for
    some $|\beta|\ll q^{-\frac{1}{2}}$ (by Lemma \ref{lem:equidistribute_rred_class}, 
    for example), we have the pointwise bound 
    \begin{equation}
        \frac{1}{\sqrt{Y}}\sum_{n\le YY'} \lambda_f(n)e(n\alpha)
        \mc I_fw\pfrc{n}{Y}\ll 1. 
        \label{eq:post_voronoi_bound}
    \end{equation}

    Therefore, we have that 
    \begin{multline*}
        \bigg|\frac{1}{\sqrt{Y}}\sum_{n\le YY'} \lambda_f(n)e\prn{\frac{an}{q}}
        \mc I_fw\pfrc{n}{Y}\bigg|^s\\
        = \bigg|\frac{1}{\sqrt{Y}}\sum_{n\le YY'} \lambda_f(n)e\prn{\frac{an}{q} + \beta n}
        \mc I_fw\pfrc{n}{Y}\bigg|^s + O((YY')^{-3}). 
    \end{multline*}

    Then, proceeding as in the proof of Lemma \ref{lem:coprime_riemann_sum}, we obtain that 
    \begin{multline*}
        \frac{1}{\varphi(q)}\sumCp_{a(q)}
        \bigg|\frac{1}{\sqrt{Y}}\sum_{n\le YY'} \lambda_f(n)e\prn{\frac{an}{q}}
        \mc I_fw\pfrc{n}{Y}\bigg|^s\\ 
        = \frac{q^{\frac{1}{2}}}{\varphi(q)}\int_0^1 
        \bigg|\frac{1}{\sqrt{Y}}\sum_{n\le YY'} \lambda_f(n)e(n\alpha)
        \mc I_fw\pfrc{n}{Y}\bigg|^s N_q([q\alpha, q\alpha + q^{\frac{1}{2}}])
        d\alpha+ O( (YY')^{-3})\\
        = \int_0^1\bigg|\frac{1}{\sqrt{Y}}\sum_{n\le YY'} \lambda_f(n)e(n\alpha)
        \mc I_fw\pfrc{n}{Y}\bigg|^s + O( (YY')^{-3}).
    \end{multline*}
    The desired result follows.

\end{proof}
Also, as in the previous subsection, by the pointwise bound (\ref{eq:post_voronoi_bound}), we have that for $q\sim Q, \tilde q\in [Q, 2Q]$,
\begin{align*}
    &\frac{1}{\sqrt{X_1(q/Q)^2}} \sum_{n\le X_1X^{2\delta_1}}\lambda_f(n)e(n\alpha)\mc I_fw\pfrc{n}{X_1(q/Q)^2}\\
    &= \frac{1}{\sqrt{X_1(\tilde q/Q)^2}} \sum_{n\le X_1X^{2\delta_1}}\lambda_f(n)e(n\alpha)\mc I_fw\pfrc{n}{X_1(\tilde q/Q)^2}\\
    &\hspace{2cm}+ O\bigg(|q - \tilde q|\sqrt{X_1}X^{5\delta_1}Q^{-1}\bigg).
\end{align*}
Then, we obtain
\begin{align*}
    \int_0^1 
    \bigg|&\frac{1}{\sqrt{X_1(q/Q)^2}}\sum_{n\le X_1X^{2\delta_1}}\lambda_f(n)e(n\alpha)\mc I_fw\pfrc{n}{X_1(q/Q)^2}\bigg|^sd\alpha\\
    &= 
    \int_ 0^1
    \bigg|\frac{1}{\sqrt{X_1(\tilde q/Q)^2}}\sum_{n\le X_1X^{2\delta_1}}\lambda_f(n)e(n\alpha)\mc I_fw\pfrc{n}{X_1(\tilde q/Q)^2}
    \bigg|^sd\alpha + O(X^{-10\delta}).
\end{align*}
Combining with Lemma \ref{lem:coprime_riemann_sum_high} exactly as in the previous section, 
we obtain that
\begin{multline*}
    \int_0^1\bigg|\frac{1}{\sqrt{X'}}\sum_{n\le X'}\lambda_f(n)e(n\alpha)\bigg|^sd\alpha\\
    = \frac{2}{3}\int_1^2t\int_0^1\bigg|\frac{1}{t\sqrt{X_1}}\sum_{n}\lambda_f(n)e(n\alpha) \mc I_f w\pfrc{n}{X_1t^2}\bigg|^s d\alpha dt
    + O(X^{-\kappa})
\end{multline*}
for some $\kappa > 0$. The desired result follows.

\end{document}